\documentclass{article}

\usepackage{graphicx}
\usepackage{indentfirst}
\usepackage{amsmath,amsfonts,amsthm,amssymb}
\usepackage{mathrsfs}
\usepackage[all]{xy}

\def\co{\colon\thinspace}
\def\ond{\hspace{-2pt}\stackrel{\circ}{\nu}\hspace{-3pt}}
\DeclareMathAlphabet{\mathsfsl}{OT1}{cmss}{m}{sl}

\def\spin{\mathrm{Spin}^c}
\def\relspin{\underline{\mathrm{Spin}^c}}

\newtheorem{thm}{Theorem}[section]
\newtheorem{lem}[thm]{Lemma}
\newtheorem{conj}[thm]{Conjecture}

\newtheorem{prop}[thm]{Proposition}
\newtheorem*{thm*}{Theorem}

\theoremstyle{definition}
\newtheorem{defn}[thm]{Definition}
\newtheorem{rem}[thm]{Remark}
\newtheorem{construction}[thm]{Construction}

\begin{document}

\title{Thurston norm and cosmetic surgeries}

\author{{Yi NI}\\{\normalsize Department of Mathematics, Caltech, MC 253-37}\\
{\normalsize 1200 E California Blvd, Pasadena, CA
91125}\\{\small\it Emai\/l\/:\quad\rm yni@caltech.edu}}

\date{}
\maketitle

\begin{abstract}
Two Dehn surgeries on a knot are called cosmetic if they yield
homeomorphic manifolds. For a null-homologous knot with certain
conditions on the Thurston norm of the ambient manifold, if the
knot admits cosmetic surgeries, then the surgery coefficients are
equal up to sign.
\end{abstract}

\section{Introduction}

Heegaard Floer homology is a powerful theory introduced by
Ozsv\'ath and Szab\'o \cite{OSzAnn1}. One important aspect of
Heegaard Floer homology is that it behaves well under Dehn
surgeries. In fact, if one knows about the knot Floer complex of a
knot, then one can compute the Heegaard Floer homology of any
surgery on the knot \cite{OSzKnot,RasThesis,OSzRatSurg}. This
makes Heegaard Floer homology very useful in the study of Dehn
surgery.

In this paper, we will use Heegaard Floer homology to study
cosmetic surgeries. We first recall the definition of cosmetic
surgeries.

\begin{defn}
If two Dehn surgeries on a knot yield homeomorphic manifolds, then
these two surgeries are {\it cosmetic}.
\end{defn}

Cosmetic surgeries are very rare. More precisely, one has the
following Cosmetic Surgery Conjecture.

\begin{conj}{\rm\cite[Problem~1.81]{Kirby}}
Suppose $K$ is a knot in a closed manifold $Y$. If the complement
of $K$ is irreducible and is not the solid torus, then any two
surgeries on $K$ do not yield manifolds which are homeomorphic via
an orientation preserving homeomorphism.
\end{conj}

The main theorem of this paper is an analogue of
\cite[Theorem~9.7]{OSzRatSurg} and \cite[Theorem~1.5]{NiNSSphere}.
See also \cite{Wu}.

All manifolds in this paper are oriented, unless otherwise stated.

\begin{thm}\label{thm:NormCos}
Suppose $Y$ is a closed $3$--manifold with $b_1(Y)>0$. Let $K$ be
a null-homologous knot in $Y$, then the inclusion map $Y-K\to Y$
induces an isomorphism $H_2(Y-K)\cong H_2(Y)$, so we can identify
$H_2(Y)$ with $H_2(Y-K)$. Suppose $r\in\mathbb Q\cup\{\infty\}$,
let $Y_r(K)$ be the manifold obtained by $r$--surgery on $K$.
Suppose $(Y,K)$ satisfies that
\begin{equation}\label{eq:LargerNorm}x_Y(h)<x_{Y-K}(h),\quad \text{for any nonzero
element}\quad h\in H_2(Y).\end{equation} Here $x_M$ is the
Thurston norm \cite{Th} in $M$. The conclusion is, if two rational
numbers $r,s$ satisfy that $Y_{r}(K)\cong\pm Y_{s}(K)$, then
$r=\pm s$.
\end{thm}

Sometimes the condition (\ref{eq:LargerNorm}) can be weakened if
there is a certain additional condition. For example, we can prove
the following theorem.

\begin{thm}\label{thm:ZeroNormCos}
Suppose $Y$ is a closed $3$--manifold with $b_1(Y)>0$. Suppose $K$
is a null-homologous knot in $Y$. Suppose $x_Y\equiv0$, while the
restriction of $x_{Y-K}$ on $H_2(Y)$ is nonzero. Then we have the
same conclusion as Theorem~\ref{thm:NormCos}. Namely, if two
rational numbers $r,s$ satisfy that $Y_{r}(K)\cong\pm Y_{s}(K)$,
then $r=\pm s$.
\end{thm}

\noindent{\bf Acknowledgements.}\quad The author is partially
supported by an AIM Five-Year Fellowship and NSF grant number
DMS-0805807.

\section{Non-triviality theorems}

In this section, we will state some non-triviality theorems in
Heegaard Floer homology. We first set up some notations we will
use in this paper.

Let $Y$ be a closed $3$--manifold. Suppose $\mathfrak S$ is a
subset of $\spin(Y)$, let
$$HF^{\circ}(Y,\mathfrak S)=\bigoplus_{\mathfrak s\in\mathfrak S}HF^{\circ}(Y,\mathfrak s),$$
where $HF^{\circ}$ is one of
$\widehat{HF},HF^{\infty},HF^{+},HF^-$. Furthermore, if $h\in
H_2(Y)$, then
$$HF^{\circ}(Y,h,i)=\bigoplus_{\mathfrak s\in\spin(Y),\langle c_1(\mathfrak s),h\rangle=2i}HF^{\circ}(Y,\mathfrak s).$$
Similarly, if $F$ is a Seifert surface for a knot $K\subset Y$,
then
$$\widehat{HFK}(Y,K,[F],i)=\bigoplus_{\xi\in\underline{\mathrm{Spin}^c}(Y,K)
,\langle c_1(\xi),\widehat F\rangle=2i}\widehat{HFK}(Y,K,\xi),$$
see \cite{OSzKnot} for more details. Following Kronheimer and
Mrowka \cite{KMsuture}, let
$$HF^{\circ}(Y|h)=HF^{\circ}(Y,h,\frac12x(h)).$$

A very important feature of Heegaard Floer homology is that it
detects the Thurston norm of a $3$--manifold. In \cite{OSzAnn2},
this result is stated for universally twisted Heegaard Floer
homology. Nevertheless, this result should also hold if one uses
untwisted coefficients. In fact, the analogous result for Monopole
Floer homology is stated with untwisted coefficients
\cite[Corollary~41.4.2]{KMBook}. In order to state our results, we
first recall two definitions.

\begin{defn}
Suppose $M$ is a compact $3$--manifold, a properly embedded
surface $S\subset M$ is {\it taut} if $x(S)=x([S])$ in
$H_2(M,\partial S)$, no proper subsurface of $S$ is
null-homologous, and if any component of $S$ lies in a homology
class that is represented by an embedded sphere then this
component is a sphere. Here $x(\cdot)$ is the Thurston norm.
\end{defn}

\begin{defn}
Suppose $K$ is a null-homologous knot in a closed $3$--manifold
$Y$. An oriented surface $F\subset Y$ is a {\it Seifert-like
surface} for $K$, if $\partial F=K$. When $F$ is connected, we say
that $F$ is a {\it Seifert surface} for $K$. We also view a
Seifert-like surface as a proper surface in $Y-\ond(K)$.
\end{defn}

As in the proof of \cite[Theorem~2.2]{HN}, using the known
non-triviality results for twisted coefficients stated in
\cite{NiNSSphere} and the Universal Coefficients Theorem, we can
prove the following theorems. (The same results can also be proved
via the approach taken in \cite{Ju,KMsuture}.)

\begin{thm}\label{thm:UnTwistNorm}
Suppose $Y$ is a closed $3$--manifold, $h\in H_2(Y)$, then
$${HF^+}(Y|h)\otimes\mathbb Q\ne0, \quad {\widehat{HF}}(Y|h)\otimes\mathbb Q\ne0.$$
\end{thm}

\begin{thm}
Suppose $K$ is a null-homologous knot in a closed 3--manifold $Y$.
Let $F$ be a taut Seifert-like surface for $K$. Then
$${\widehat{HFK}}(Y,K,[F],\frac{x(F)+1}2)\otimes\mathbb Q\ne0.$$
\end{thm}

\section{A surgery formula}

Suppose $K\subset Y$ is a null-homologous knot. Let $Y_{p/q}(K)$
denote the manifold obtained by $\frac pq$--surgery on $K$. Note
that there is a natural identification
$$\mathrm{Spin}^c(Y_{p/q}(K))\cong\mathrm{Spin}^c(Y)\times\mathbb Z/p\mathbb Z.$$
Let $\pi\co\mathrm{Spin}^c(Y_{p/q}(K))\to\mathrm{Spin}^c(Y)$ be
the projection to the first factor.

The goal of this section is to prove the following theorem, which
is a (much easier) analogue of \cite[Theorem~1.1]{OSzRatSurg}.

\begin{thm}\label{thm:SurgForm}
Suppose $K\subset Y$ is a null-homologous knot. If
$\widehat{HF}(Y,\mathfrak s)=0$, then there exists a constant
$C=C(Y,K,\mathfrak s)$, such that
$$\mathrm{rank}\:\widehat{HF}(Y_{p/q}(K),\pi^{-1}(\mathfrak s))=qC.$$
\end{thm}

\subsection{Large surgeries on rationally null-homologous knots}

Suppose $K\subset Y$ is a rationally null-homologous knot. We
construct a Heegaard diagram
$(\Sigma,\mbox{\boldmath$\alpha$},\mbox{\boldmath$\beta$},w,z)$
for $(Y,K)$, such that $\beta_1=\mu$ is a meridian of $K$.
Moreover, $w,z$ are two base points associated with a marked point
on $\beta_1$ as in \cite{OSzKnot}. There is a curve
$\lambda\subset \Sigma$ which gives rise to the knot $K$. Doing
oriented cut-and-pastes to $\lambda$ and $m$ parallel copies of
$\mu$, we get a connected simple closed curve supported in a small
neighborhood of $\mu\cup\lambda$. We often denote this curve by
$m\mu+\lambda$. The $m$ parallel copies of $\mu$ are supported in
a small neighborhood of $\mu$. We call this neighborhood the {\it
winding region} for $m\mu+\lambda$.
$(\Sigma,\mbox{\boldmath$\alpha$},\mbox{\boldmath$\gamma$},z)$ is
a diagram for $Y_{m\mu+\lambda}(K)$, where $\gamma_1=m\mu+\lambda$
and all other $\gamma_i$'s are small Hamiltonian translations of
$\beta_i$'s.

\begin{defn}\label{defn:Xi}
As in \cite[Section~4]{OSzRatSurg}, one defines a map
$$\Xi\co\mathrm{Spin}^c(Y_{m\mu+\lambda}(K))\to\underline{\mathrm{Spin}^c}(Y,K)$$
as follows. If $\mathfrak
t\in\mathrm{Spin}^c(Y_{m\mu+\lambda}(K))$ is represented by a
 point $\mathbf y$ supported in
the winding region, let $\mathbf x\in\mathbb T_{\alpha}\cap\mathbb
T_{\beta}$ be the ``nearest point", and let $\psi\in\pi_2(\mathbf
y,\Theta,\mathbf x)$ be a small triangle. Then
\begin{equation}\label{eq:Xi}
\Xi(\mathfrak t)=\underline{\mathfrak s}_{w,z}(\mathbf
x)+\big(n_w(\psi)-n_z(\psi)\big)\cdot\mu.
\end{equation}
\end{defn}

When we construct the Heegaard triple diagram
$$(\Sigma,\mbox{\boldmath$\alpha$},\mbox{\boldmath$\beta$},\mbox{\boldmath$\gamma$},w,z),$$
the position of the meridian $\beta_1$ relative to the
 points in $\lambda\cap\gamma_1$ may vary. Our next lemma
says that the choice of the position of $\beta_1$ does not affect
the definition of $\Xi$.

\begin{lem}
Suppose we have two Heegaard triple diagrams as above
$$\Gamma_1=(\Sigma,\mbox{\boldmath$\alpha$},\mbox{\boldmath$\beta$}^1,\mbox{\boldmath$\gamma$},w^1,z^1),\quad
\Gamma_2=(\Sigma,\mbox{\boldmath$\alpha$},\mbox{\boldmath$\beta$}^2,\mbox{\boldmath$\gamma$},w^2,z^2).$$
The two sets $\mbox{\boldmath$\beta$}^1$ and
$\mbox{\boldmath$\beta$}^2$ differ at the meridian, where the
meridian $\beta_1^2\in\mbox{\boldmath$\beta$}^2$ is a parallel
translation of the meridian
$\beta_1^1\in\mbox{\boldmath$\beta$}^1$, still supported in the
winding region. The two base points are moved together with the
meridian.

Using these two diagrams, we can define two maps
$$\Xi^1,\Xi^2\co\mathrm{Spin}^c(Y_{m\mu+\lambda}(K))\to\underline{\mathrm{Spin}^c}(Y,K).$$
Then $\Xi^1=\Xi^2$.
\end{lem}
\begin{proof}
Without loss of generality, we may assume there is only one
intersection point of $\lambda\cap\gamma_1$ between $\beta^1_1$
and $\beta^2_1$. See Figure~\ref{fig:Xi} for an illustration.

Suppose $\mathbf y^1,\mathbf y^2\in\mathbb T_{\alpha}\cap\mathbb
T_{\gamma}$ are two intersection points supported in the winding
region, and suppose their $\gamma_1$--coordinates are $y^1,y^2$,
respectively. Assume $\mathfrak s_{w^1}(\mathbf y^1)=\mathfrak
s_{w^2}(\mathbf y^2)=\mathfrak t$, we want to prove that
$\Xi^1(\mathfrak t)=\Xi^2(\mathfrak t)$.

By \cite[Lemma~2.19]{OSzAnn1},
\begin{eqnarray*}
\mathfrak s_{w^1}(\mathbf y^1)-\mathfrak s_{w^1}(\mathbf
y^2)&=&\mathrm{PD}(\varepsilon(\mathbf y^2,\mathbf y^1)),\\
\mathfrak s_{w^2}(\mathbf y^2)-\mathfrak s_{w^1}(\mathbf
y^2)&=&\mathrm{PD}(\mu).
\end{eqnarray*}
Hence $\varepsilon(\mathbf y^2,\mathbf y^1)=\mu$. Let
$\tilde{\mathbf y}^1\in\mathbb T_{\alpha}\cap\mathbb T_{\gamma}$
be the point whose coordinates coincide with the coordinates of
$\mathbf y^1$, except that its $\gamma_1$--coordinate is the next
intersection point to $y^1$ on the same $\alpha$--curve, denoted
$\tilde y^1$. Then $\varepsilon(\tilde{\mathbf y}^1,\mathbf
y^1)=\mu$, so $\tilde{\mathbf y}^1$ is in the same equivalence
class as $\mathbf y^2$.

Now we only need to prove that
\begin{equation}\label{eq:XiShift}
\Xi^1(\mathfrak s_{w^1}(\mathbf y^1))=\Xi^2(\mathfrak
s_{w^2}(\tilde{\mathbf y}^1)).
\end{equation}
Let $\mathbf x^1\in\mathbb T_{\alpha}\cap\mathbb
T_{\beta^1},\tilde{\mathbf x}^1\in\mathbb T_{\alpha}\cap\mathbb
T_{\beta^2}$ be the nearest points to $\mathbf y^1,\tilde{\mathbf
y}^1$, respectively. It is clear that $\underline{\mathfrak
s}_{w^1,z^1}(\mathbf x^1)=\underline{\mathfrak
s}_{w^2,z^2}(\tilde{\mathbf x}^1)$. Moreover, the small triangle
for $\tilde{\mathbf y}^1$ in $\Gamma_2$ is just a translation of
the small triangle for ${\mathbf y}^1$ in $\Gamma_1$, so they
contribute the same $n_w(\psi)-n_z(\psi)$ term in (\ref{eq:Xi}).
So (\ref{eq:XiShift}) follows.
\end{proof}

\begin{figure}
\begin{picture}(340,120)
\put(0,0){\scalebox{0.64}{\includegraphics*[25pt,345pt][560pt,
530pt]{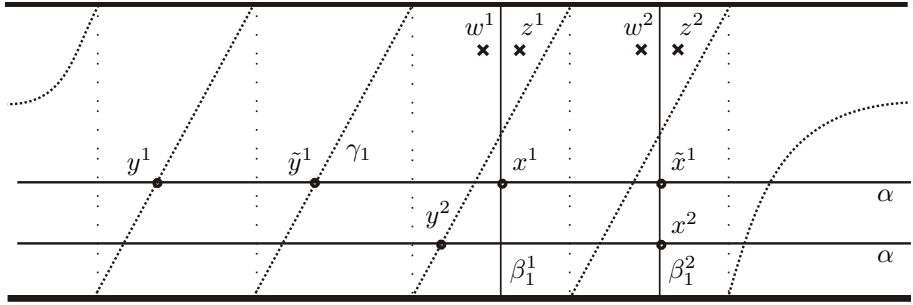}}}

\put(330,43){$\alpha$}

\put(330,20){$\alpha$}

\put(191,14){$\beta^1_1$}

\put(251,14){$\beta^2_1$}

\put(174,106){$w^1$}

\put(195,106){$z^1$}

\put(192,55){$x^1$}

\put(252,31){$x^2$}

\put(252,55){$\tilde x^1$}

\put(234,106){$w^2$}

\put(255,106){$z^2$}

\put(159,35){$y^2$}

\put(46,55){$y^1$}

\put(107,55){$\tilde y^1$}

\put(129,60){$\gamma_1$}

\end{picture}
\caption{Local picture of the two triple Heegaard
diagrams}\label{fig:Xi}
\end{figure}

\begin{rem}
In \cite{OSzRatSurg}, in order to define $\Xi(\mathfrak t)$, one
places the meridian in a position such that the equivalence class
of intersection points representing $\mathfrak t$ is supported in
the winding region. The above lemma removes this restriction.
\end{rem}

\begin{lem}\label{lem:XiSurj}
Suppose $\xi\in\relspin(Y,K)$. For all sufficiently large $m$,
there exists $\mathfrak t\in\spin(Y_{m\mu+\lambda}(K))$, such that
$\Xi(\mathfrak t)=\xi$.
\end{lem}
\begin{proof}
Let $\mathfrak s\in\spin(Y)$ be the underlying Spin$^c$ structure
of $\xi$. We can choose a Heegaard diagram for $(Y,K)$ such that
some $\mathbf x\in\mathbb T_{\alpha}\cap\mathbb T_{\beta}$
represents $\mathfrak s$, then $\xi=\underline{\mathfrak
s}_{w,z}(\mathbf x)+n\cdot\mu$ for some $n\in\mathbb Z$. Now our
desired result follows from the definition of $\Xi$.
\end{proof}

The following proposition is a part of
\cite[Theorem~4.1]{OSzRatSurg}.

\begin{prop}\label{prop:LargeSurg}
Let $K\subset Y$ be a rationally null-homologous knot in a closed,
oriented three-manifold, equipped with a framing $\lambda$. Let
$$\widehat A_{\xi}(Y,K
)=C_{\xi}\big\{\max\{i,j\}=0\big\},$$ where
$C_{\xi}=CFK^{\infty}(Y,K,\xi)$ as in \cite{OSzRatSurg}. Then, for
all sufficiently large $m$ and all $\mathfrak
t\in\mathrm{Spin}^c(Y_{m\mu+\lambda}(K))$, there is an isomorphism
$$\Psi_{\mathfrak t,m}\co \widehat{CF}(Y_{m\mu+\lambda}(K),\mathfrak t)
\to\widehat A_{\Xi(\mathfrak t)}(Y,K).$$
\end{prop}

\subsection{Rational surgeries on null-homologous knots}

Let $K$ be a null-homologous knot in $Y$. As in
\cite[Section~7]{OSzRatSurg}, $Y_{\frac pq}(K)$ can be realized by
a Morse surgery with coefficient $a$ on the knot
$K'=K\#O_{q/r}\subset Y'=Y\#L(q,r)$, where $O_{q/r}$ is a
$U$--knot in $L(q,r)$, $p=aq+r$. Let
$$\Xi'\co\spin(Y'_{a\mu'+\lambda'})\to\relspin(Y',K')$$ be the map
defined in Definition~\ref{defn:Xi}.

\begin{figure}[!htbp]
\begin{center}
\begin{picture}(345,139)
\put(0,0){\scalebox{0.59}{\includegraphics*[5pt,425pt][585pt,
660pt]{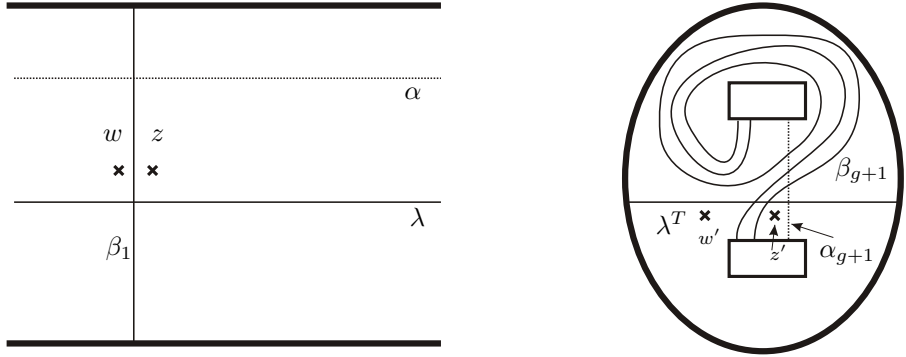}}}

\put(316,69){$\beta_{g+1}$}

\put(153,98){$\alpha$}

\put(40,39){$\beta_1$}

\put(39,82){$w$}

\put(57,82){$z$}

\put(264,43){$\scriptstyle w'$}

\put(291,37){$\scriptstyle z'$}

\put(310,38){$\alpha_{g+1}$}

\put(248,47){$\lambda^T$}

\put(155,50){$\lambda$}
\end{picture}
\caption{The left hand side is a piece of a Heegaard diagram for
$(Y,K)$. The right hand side is a genus $1$ Heegaard diagram for
$(L(q,r),O_{q/r})$. The boundary of the oval is capped off with a
disk, and the boundaries of the two rectangles are glued together
via a reflection. Here we choose $q=3,r=2$.}\label{fig:TwoKnots}
\end{center}
\end{figure}

\begin{construction}\label{constr:ConnSum}
Let
$$(\Sigma,\mbox{\boldmath${\alpha}$}=\{\alpha_1,\dots,\alpha_g\},
\mbox{\boldmath${\beta}$}=\{\beta_1,\dots,\beta_g\}, w,z)$$ be a
doubly-pointed Heegaard diagram for $(Y,K)$, such that $\beta_1$
is a meridian for $K$ and the two base points are induced from a
marked point on $\beta_1$. Suppose $\lambda\subset\Sigma$
represents a longitude of $K$.

Let $$(T,\{\alpha_{g+1}\},\{\beta_{g+1}\},w',z')$$ be a genus $1$
Heegaard diagram for $(L(q,r),O_{q/r})$. As in
Figure~\ref{fig:TwoKnots}, $\beta_{g+1}$ intersects $\alpha_{g+1}$
exactly $q$ times and intersects the boundary of each rectangle
exactly $r$ times. Suppose $\lambda^T\subset T$ represents a
longitude of $O_{q/r}$.

We perform the connected sum of $\Sigma$ and $T$ by identifying
the neighborhoods of $z$ and $w'$, hence we get a new genus
$(g+1)$ surface $\Sigma'$. Then
$$(\Sigma',\mbox{\boldmath${\alpha}$}'=\mbox{\boldmath${\alpha}$}\cup\{\alpha_{g+1}\},
\mbox{\boldmath${\beta}$}'=\mbox{\boldmath${\beta}$}\cup\{\beta_{g+1}\},w,z')$$
is a Heegaard diagram for $(Y',K')$. The longitude $\lambda'$ of
$K'$ is a connected sum of $\lambda$ and $\lambda^T$. \qed
\end{construction}

We define $$\Pi_1\co\relspin(Y',K')\to\relspin(Y,K)$$ as follows.
Given $\xi'\in\relspin(Y',K')$, suppose $\mathbf x'\in\mathbb
T_{\alpha'}\cap\mathbb T_{\beta'}$ represents the underlying
Spin$^c$ structure of $\xi'$, then
$$\xi'=\underline{\mathfrak s}_{w,z'}(\mathbf x')+n\cdot\mu'$$
for some $n\in\mathbb Z$. Now let $\mathbf x$ be the projection of
$\mathbf x'$ to $\mathbb T_{\alpha}\cap\mathbb T_{\beta}$, then
$$\Pi_1(\xi')=\underline{\mathfrak s}_{w,z}(\mathbf x)+n\cdot\mu.$$

The following proposition is obvious. (See also
\cite[Corollary~5.3]{OSzRatSurg}.)

\begin{prop}\label{prop:Kunneth}
For any $\xi'\in\relspin(Y',K')$, we have
$${CFK^{\infty}}(Y',K',\xi')\cong{CFK^{\infty}}(Y,K,\Pi_1(\xi'))$$
as $\mathbb Z\oplus\mathbb Z$--filtered chain complexes.
\end{prop}

\begin{lem}\label{lem:piDecomp}
When $m$ is sufficiently large, we have
$$\pi=G_{Y,K}\circ\Pi_1\circ\Xi'.$$ Here $G_{Y,K}\co\relspin(Y,K)\to\spin(Y)$ is the map defined in
\cite[Section~2.2]{OSzRatSurg}.
\end{lem}

\begin{figure}[!htbp]
\begin{center}
\begin{picture}(345,139)
\put(0,0){\scalebox{0.59}{\includegraphics*[5pt,425pt][585pt,
660pt]{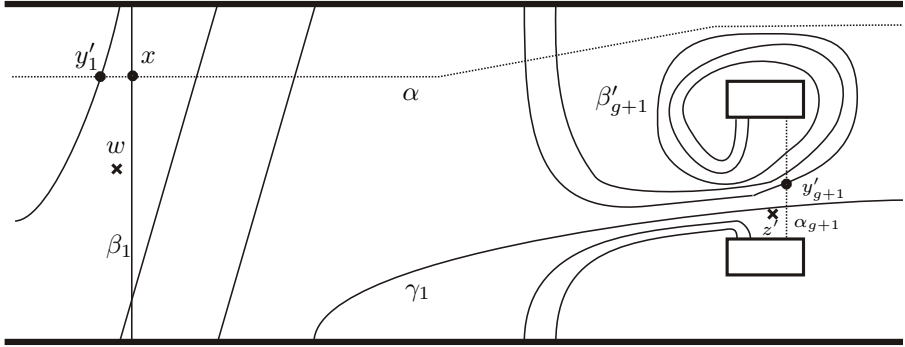}}}

\put(154,23){$\gamma_{1}$}

\put(153,97){$\alpha$}

\put(41,39){$\beta_1$}

\put(41,77){$w$}

\put(29,111){$y'_1$}

\put(304,62){$\scriptstyle y'_{g+1}$}

\put(289,46){$\scriptstyle z'$}

\put(301,49){$\scriptstyle \alpha_{g+1}$}

\put(54,111){$x$}

\put(226,95){$\beta'_{g+1}$}
\end{picture}
\caption{A Heegaard diagram for $Y'_{a\mu'+\lambda'}(K')$. Here we
choose $a=3$.}\label{fig:HS1}
\end{center}
\end{figure}

\begin{proof}
We follow the notation in Construction~\ref{constr:ConnSum}. Since
$\lambda'$ intersects $\beta_1$ exactly once, we can slide
$\beta_{g+1}$ over $\beta_1$ $r$ times to eliminate the
intersection points in $\beta_{g+1}\cap\lambda'$. The new curve is
denoted $\beta'_{g+1}$ as in Figure~\ref{fig:HS1}. Then
$$(\Sigma',\mbox{\boldmath${\alpha}$}',
\mbox{\boldmath${\beta}$}''=\mbox{\boldmath${\beta}$}\cup\{\beta'_{g+1}\},w,z')$$
is also a Heegaard diagram for $(Y',K')$. Let
$\gamma_{1}=a\beta_1+\lambda'$, then
$$(\Sigma',\mbox{\boldmath${\alpha}$}',
\mbox{\boldmath${\gamma}$}_1=\{\gamma_1,\beta_2,\dots,\beta_g,\beta'_{g+1}\},
w)$$ is a Heegaard diagram for $Y'_{a\mu'+\lambda'}(K')$.

\begin{figure}[!htbp]
\begin{center}
\begin{picture}(345,139)
\put(0,0){\scalebox{0.59}{\includegraphics*[5pt,425pt][585pt,
660pt]{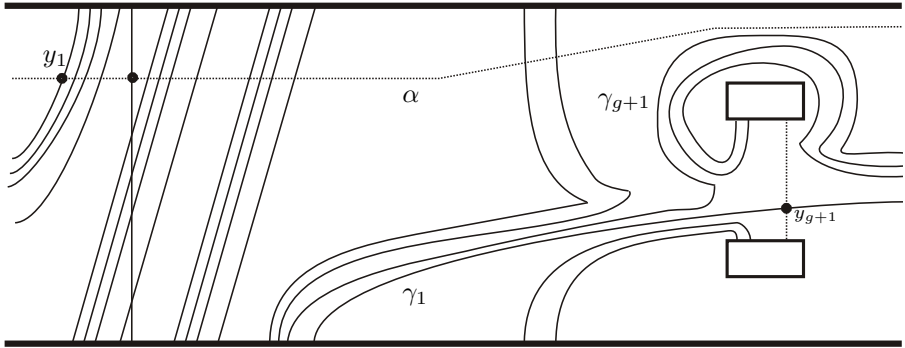}}}

\put(227,96){$\gamma_{g+1}$}

\put(153,22){$\gamma_{1}$}

\put(153,97){$\alpha$}

\put(17,112){$y_1$}

\put(301,53){$\scriptstyle y_{g+1}$}
\end{picture}
\caption{After $q$ handleslides, we get a Heegaard diagram for
$Y_{p/q}(K)$.}\label{fig:HS2}
\end{center}
\end{figure}

The curve $\alpha_{g+1}$ intersects $\gamma_{1}$ exactly once. We
can slide $\beta'_{g+1}$ over $\gamma_{1}$ $q$ times to eliminate
its $q$ intersection points with $\alpha_{g+1}$. The new curve is
denoted $\gamma_{g+1}$ as in Figure~\ref{fig:HS2}.  Now
$$(\Sigma',\mbox{\boldmath${\alpha}$}',
\mbox{\boldmath${\gamma}$}_2=\{\gamma_1,\beta_2,\dots,\beta_g,\gamma_{g+1}\},
w)$$ is a Heegaard diagram for
$Y'_{a\mu'+\lambda'}(K')=Y_{p/q}(K)$. Moreover, we may slide other
$\alpha$--curves over $\alpha_{g+1}$ to eliminate their
intersection points with $\gamma_{1}$. A destabilization will
remove $\alpha_{g+1}$ and $\gamma_1$. Now we get a diagram
$$(\Sigma^*,\mbox{\boldmath${\alpha}$}^*,
\mbox{\boldmath${\gamma}$}^*, w)$$
 which is isomorphic to
$$(\Sigma,\mbox{\boldmath${\alpha}$},
\{\beta_2,\dots,\beta_g,\gamma_{g+1}^*\}, w),$$ where
$\gamma_{g+1}^*$ is the image of $\gamma_{g+1}$ under the
destabilization.

We want to show that $\gamma_{g+1}^*$ is isotopic to
$p\mu+q\lambda$, the curve obtained by doing cut-and-pastes to $p$
parallel copies of $\mu$ and $q$ parallel copies of $\lambda$. In
fact, $\gamma_{g+1}^*$ is supported in a small neighborhood of
$\mu\cup\lambda$, so it must be isotopic to $p'\mu+q'\lambda$ for
some $p',q'$. It is easy to compute the intersection numbers of
$\gamma_{g+1}$ with $\lambda$ and $\mu=\beta_1$, which are
$p=aq+r$ and $q$. The intersection numbers of $\gamma_{g+1}^*$
with $\mu$ and $\lambda$ remains the same, so
$\gamma_{g+1}^*=p\mu+q\lambda$.

Suppose $\mathfrak t\in\spin(Y_{p/q}(K))$. We want to prove
\begin{equation}\label{eq:piFactor}
\pi(\mathfrak t)=G_{Y,K}\circ\Pi_1\circ\Xi'(\mathfrak t).
\end{equation}

We first consider the right hand side of (\ref{eq:piFactor}). Let
$\mathbf y'$ be a point in $\mathbb T_{\alpha'}\cap\mathbb
T_{\gamma_1}$ which is supported in the winding region and
represents $\mathfrak t$ (Figure~\ref{fig:HS1}). Suppose the
$\gamma_{1}$--coordinate of $\mathbf y'$ is $y'_1$ and the
$\beta'_{g+1}$--coordinate is $y'_{g+1}$.

Let $\mathbf x'\in \mathbb T_{\alpha'}\cap\mathbb T_{\beta''}$ be
the nearest point to $\mathbf y'$, then (\ref{eq:Xi}) implies that
$$\Xi'(\mathfrak t)=\underline{\mathfrak s}_{w,z'}(\mathbf
x')+n\cdot\mu'$$ for some $n\in\mathbb Z$. Let $\mathbf x$ be the
projection of $\mathbf x'$ to $\mathbb T_{\alpha}\cap\mathbb
T_{\beta}$, then $$\Pi_1\circ\Xi'(\mathfrak
t)=\underline{\mathfrak s}_{w,z}(\mathbf x)+n\cdot\mu.$$ Hence
$$G_{Y,K}\circ\Pi_1\circ\Xi'(\mathfrak t)=\mathfrak s_w(\mathbf x).$$

Now we consider the left hand side of (\ref{eq:piFactor}). As in
Figure~\ref{fig:HS2}, we get another Heegaard diagram for
$Y_{p/q}(K)$ by $q$ handle slides. In this diagram, we can find a
point $\mathbf y\in\mathbb T_{\alpha'}\cap\mathbb T_{\gamma_2}$
which represents $\mathfrak t$ as $\mathbf y'$ does. In fact,
since $\alpha_{g+1}$ intersects $\gamma_{1}$ exactly once and is
disjoint from other $\gamma$--curves, $\mathbf y$ must contain the
intersection point of $\alpha_{g+1}$ and $\gamma_{1}$, denoted
$y_{g+1}$. The $\gamma_1$--coordinate of $\mathbf y$, called
$y_1$, is determined by $y'_1$ and $y'_{g+1}$: it is one of the
$q$ intersection points on $\gamma_{g+1}$ near $y'_1$, and the
choice among these $q$ points is specified by the position of
$y'_{g+1}$. Other coordinates of $\mathbf y$ are the same as
$\mathbf y'$.

After handleslides and one destabilization, we get a point
$\mathbf y^*\in\mathbb T_{\alpha^*}\cap\mathbb T_{\gamma^*}$ whose
coordinates are the same as $\mathbf x$ except that its
$\gamma_1$--coordinate is $y_1$. So its nearest point in $\mathbb
T_{\alpha}\cap\mathbb T_{\beta}$ is $\mathbf x$, hence $\mathbf x$
represents $\pi(\mathfrak t)$. This proves (\ref{eq:piFactor}).
\end{proof}

\begin{lem}\label{lem:Ambient} Let
$H(\widehat{A}_{\xi}(Y,K))$ be the homology of the chain complex
$\widehat{A}_{\xi}(Y,K)$. For a fixed $\xi$, when $|n|\gg0$,
$$H(\widehat{A}_{\xi+n\cdot\mu}(Y,K))\cong\widehat{HF}(Y, G_{Y,K}(\xi)).$$
\end{lem}
\begin{proof}
By the definitions \begin{eqnarray*}
\widehat{A}_{\xi+n\cdot\mu}(Y,K)&=&C_{\xi+n\cdot\mu}\left\{\max\{i,j\}=0\right\}\\
&=&C_{\xi}\left\{\max\{i,j-n\}=0\right\}.
\end{eqnarray*}
By the adjunction inequality, $H(C_{\xi}\{i,j\})=0$ when
$|i-j|\gg0$. So $$H(C_{\xi}\left\{\max\{i,j-n\}=0\right\})\cong
H(C_{\xi}\{i=0\})$$ when $n\gg0$. The latter group is isomorphic
to $\widehat{HF}(Y, G_{Y,K}(\xi))$ by
\cite[Proposition~3.2]{OSzRatSurg}.

When $n\ll0$, we have
$$H(C_{\xi}\left\{\max\{i,j-n\}=0\right\})\cong
H(C_{\xi}\{j=n\})\cong H(C_{\xi}\{j=0\}),$$ which is isomorphic to
$\widehat{HF}(Y, G_{Y,-K}(\xi))$ by
\cite[Proposition~3.2]{OSzRatSurg}. Now by \cite[Equation
(4)]{OSzRatSurg} and the fact that $K$ is null-homologous, we have
$G_{Y,K}(\xi)=G_{Y,-K}(\xi)$.
\end{proof}

\begin{lem}\label{lem:FinNonzero}
Suppose $\widehat{HF}(Y,\mathfrak s)=0$, then
$H(\widehat{A}_{\xi'}(Y',K'))\ne0$ for only finitely many
$\xi'\in(G_{Y,K}\circ\Pi_1)^{-1}(\mathfrak s)$.
\end{lem}
\begin{proof}
For each $\xi\in\relspin(Y,K)$, there are exactly $q$ relative
Spin$^c$ structures in $\Pi_1^{-1}(\xi)$. Moreover, by
Proposition~\ref{prop:Kunneth}, if $\xi'\in\Pi_1^{-1}(\xi)$, then
$$\widehat{A}_{\xi'}(Y',K')\cong\widehat{A}_{\xi}(Y,K).$$ Hence we
only need to show that $H(\widehat{A}_{\xi}(Y,K))\ne0$ for only
finitely many $\xi\in G_{Y,K}^{-1}(\mathfrak s)$.

Pick any $\xi\in G_{Y,K}^{-1}(\mathfrak s)$, then
$$G_{Y,K}^{-1}(\mathfrak s)=\{\xi+i\cdot\mu|\:i\in\mathbb Z\}.$$
By Lemma~\ref{lem:Ambient}, $H(\widehat{A}_{\xi+i\cdot\mu}(Y,K))$
is isomorphic to $\widehat{HF}(Y,\mathfrak s)$ when $|i|$ is
large, hence is $0$. This finishes the proof.
\end{proof}

\begin{prop}\label{prop:qCLarge}
When $m$ is sufficiently large,
\begin{eqnarray*}
\widehat{HF}(Y'_{m\mu'+\lambda'}(K'),\pi^{-1}(\mathfrak s))&\cong&
\bigoplus_{\{\xi'|\:G_{Y,K}\circ\Pi_1(\xi')=\mathfrak
s\}}H(\widehat{A}_{\xi'}(Y',K'))\\
&\cong&\bigoplus^q\bigoplus_{\{\xi|\:G_{Y,K}(\xi)=\mathfrak
s\}}H(\widehat{A}_{\xi}(Y,K)).
\end{eqnarray*}
\end{prop}
\begin{proof}
By Proposition~\ref{prop:LargeSurg}, when $m$ is sufficiently
large
$$\widehat{HF}(Y'_{m\mu'+\lambda'}(K'),\pi^{-1}(\mathfrak s))\cong
\bigoplus_{\mathfrak t\in\pi^{-1}(\mathfrak s
)}H(\widehat{A}_{\Xi'(\mathfrak t)}(Y',K')).$$

By Lemma~\ref{lem:piDecomp},
$$\Xi'(\pi^{-1}(\mathfrak s))=\Xi'\left(\Xi'^{-1}\circ(G_{Y,K}\circ\Pi_1)^{-1}(\mathfrak s)\right)\subset(G_{Y,K}\circ\Pi_1)^{-1}(\mathfrak s).$$
Consider the map
$$\Xi_{\mathfrak s}'\co\pi^{-1}(\mathfrak s)\to (G_{Y,K}\circ\Pi_1)^{-1}(\mathfrak s).$$
By \cite[Lemma~2.4]{NiNSSphere}, $\Xi_{\mathfrak s}'$ is
injective. Moreover, by Lemmas~\ref{lem:XiSurj} and
\ref{lem:FinNonzero}, when $m$ is sufficiently large, the range of
$\Xi_{\mathfrak s}'$ contains all
$\xi'\in(G_{Y,K}\circ\Pi_1)^{-1}(\mathfrak s)$ satisfying
$H(\widehat{A}_{\xi'}(Y',K'))\ne0$. This proves the first
equality.

In order to prove the second equality, we note that for each
$\xi\in\relspin(Y,K)$, there are exactly $q$ relative Spin$^c$
structures in $\Pi_1^{-1}(\xi)$. Moreover, by
Proposition~\ref{prop:Kunneth}, if $\xi'\in\Pi_1^{-1}(\xi)$, then
$$\widehat{A}_{\xi'}(Y',K')\cong\widehat{A}_{\xi}(Y,K).$$
So the second equality easily follows.
\end{proof}

\begin{proof}[Proof of Theorem~\ref{thm:SurgForm}]
Let $$C=\mathrm{rank}\bigoplus_{\{\xi|\:G_{Y,K}(\xi)=\mathfrak
s\}}H(\widehat{A}_{\xi}(Y,K)).$$ By
Proposition~\ref{prop:qCLarge},
$$\mathrm{rank}\:\widehat{HF}(Y_{p/q},\pi^{-1}(\mathfrak s))=qC$$
when $p$ is sufficiently large.

Since $\widehat{HF}(Y,\mathfrak s)=0$, we have
$\widehat{HF}(Y',\mathfrak s')=0$ for any $\mathfrak s'$ that
extends $\mathfrak s$. By \cite[Theorem~9.12]{OSzAnn2}, we have
the long exact sequence
$$
\begin{xymatrix}{
\widehat{HF}(Y',P_1^{-1}(\mathfrak
s))\ar[r]&\widehat{HF}(Y'_{m\mu'+\lambda'}(K'),\pi_m^{-1}(\mathfrak
s))\ar[dl]\\
\widehat{HF}(Y'_{(m+1)\mu'+\lambda'}(K'),\pi_{m+1}^{-1}(\mathfrak
s))\ar[u]&}
\end{xymatrix},
$$
where
$$P_1\co\spin(Y')\to\spin(Y),$$
$$\pi_m\co\spin(Y'_{m\mu'+\lambda'}(K'))\to\spin(Y)$$
are the natural projection maps. Since
$\widehat{HF}(Y',P_1^{-1}(\mathfrak s))=0$, we have
$$\widehat{HF}(Y'_{a\mu'+\lambda'}(K'),\pi_a^{-1}(\mathfrak s))\cong
\widehat{HF}(Y'_{m\mu'+\lambda'}(K'),\pi_m^{-1}(\mathfrak s))$$
for $m$ sufficiently large. Hence its rank is always $qC$.
\end{proof}

\section{Cosmetic surgeries}

\begin{proof}[Proof of Theorem~\ref{thm:NormCos}]
Assume there are two rational numbers
$\displaystyle\frac{p_1}{q_1},\frac{p_2}{q_2}$ satisfying that
there is a homeomorphism
$$f\co Y_{\frac{p_1}{q_1}}\to\pm Y_{\frac{p_2}{q_2}},$$ then
$|p_1|=|p_2|$ for homological reasons. If
$\displaystyle\frac{p_1}{q_1}\ne\pm\frac{p_2}{q_2}$, then we can
assume
$$0<q_1<q_2.$$

Without loss of generality, we may assume $Y-K$ is irreducible.
 By (\ref{eq:LargerNorm}) and the adjunction inequality, we
conclude that $\widehat{HF}(Y,h,\frac12x_{Y-K}(h))=0$. It then
follows from Theorem~\ref{thm:SurgForm} that there is a constant
$C_h$, such that
$$\mathrm{rank}\:\widehat{HF}(Y_{p/q}(K),h,\frac12x_{Y-K}(h))=qC_h.$$
Since (\ref{eq:LargerNorm}) holds, \cite[Corollary~2.4]{G2}
implies that
$$x_{Y-K}(h)=x_{Y_{p/q}(K)}(h)$$ for any nonzero $h\in H_2(Y)$ and $\displaystyle\frac pq\in\mathbb
Q$. Theorem~\ref{thm:UnTwistNorm} then implies that
$$\mathrm{rank}\:\widehat{HF}(Y_{p/q}(K)|h)=qC_h\ne0.$$

Since $K$ is null-homologous, the inclusion maps $Y-K\to Y_r$
induce isomorphisms on $H_2$ for each $r\in\mathbb
Q\cup\{\infty\}\backslash\{0\}$. Hence we can identify
$H_2(Y_r(K))$ with $H_2(Y)$. Now $f_*\co
H_2(Y_{\frac{p_1}{q_1}})\to H_2(Y_{\frac{p_2}{q_2}})$ can be
regarded as a map $$f_*\co H_2(Y)\to H_2(Y).$$ Fix a nonzero $h\in
H_2(Y)$, we have
$$
\mathrm{rank}\:\widehat{HF}(Y_{\frac{p_1}{q_1}}|f^n_*(h))=
\frac{q_1}{q_2}\mathrm{rank}\:\widehat{HF}(Y_{\frac{p_2}{q_2}}|f^n_*(h))\ne0
$$
for any $n\in\mathbb Z$. Moreover, since $f\co
Y_{\frac{p_1}{q_1}}\to\pm Y_{\frac{p_2}{q_2}}$ is a homeomorphism,
we have
$$
\mathrm{rank}\:\widehat{HF}(Y_{\frac{p_1}{q_1}}|f^{n-1}_*(h))=
\mathrm{rank}\:\widehat{HF}(Y_{\frac{p_2}{q_2}}|f^{n}_*(h)).
$$
Thus we get
$$\mathrm{rank}\:\widehat{HF}(Y_{\frac{p_1}{q_1}}|f^n_*(h))=\left(\frac{q_1}{q_2}\right)^n
\mathrm{rank}\:\widehat{HF}(Y_{\frac{p_1}{q_1}}|h)\ne0.$$ So
$0<\mathrm{rank}\:\widehat{HF}(Y_{\frac{p_1}{q_1}}|h)<1$
 when $n$ is sufficiently large, which is impossible.
\end{proof}

\begin{proof}[Proof of Theorem~\ref{thm:ZeroNormCos}]
Since $x_Y\equiv0$, the adjunction inequality implies that
$\widehat{HF}(Y,h,\frac12x_{Y-K}(h))=0$ for any $h\in H_2(Y)$
satisfying $x_{Y-K}(h)\ne0$. Using Theorems~\ref{thm:SurgForm},
\ref{thm:UnTwistNorm} and \cite[Corollary~2.4]{G2}, we have
$$\mathrm{rank}\:\widehat{HF}(Y_{p/q}(K)|h)=qC_h$$ for some
nonzero constant $C_h$. Now the argument is the same as in the
proof of Theorem~\ref{thm:NormCos}.
\end{proof}

\end{document}